\documentclass[a4paper,12pt]{amsart} 

\usepackage{amssymb,amsmath}

\usepackage[right=3cm, left=3cm, top=2cm, bottom=2cm]{geometry}

\newtheorem{thm}{Theorem}[section] 
\newtheorem*{thm*}{Theorem}
\newtheorem{prop}[thm]{Proposition}
\newtheorem{cor}[thm]{Corollary} 
\newtheorem{lem}[thm]{Lemma}

\theoremstyle{definition} 
\newtheorem*{defin}{Definition} 
\newtheorem{rem}[thm]{Remark}

\newtheorem*{ackn}{Acknowledgment}
 
\numberwithin{equation}{section}

\newcommand{\skal}[2]{\langle #1,#2\rangle}
\newcommand{\alg}[1]{\mathfrak{#1}}
\newcommand{\lin}[1]{\langle #1\rangle}
\newcommand{\cgtwo}[2]{\langle #1,#2\rangle_{\mathbb{C}}}

\begin{document}

\title{A note on invariant description of $SU(2)$--structures in dimension $5$}

\author{Kamil Niedzia\l omski}
\date{\today}

 
\address{
Department of Mathematics and Computer Science \endgraf
University of \L\'{o}d\'{z} \endgraf
ul. Banacha 22, 90-238 \L\'{o}d\'{z} \endgraf
Poland
}
\email{kamil.niedzialomski@wmii.uni.lodz.pl}

\begin{abstract}
We develop an invariant approach to $SU(2)$--structures on spin $5$--manifolds. We characterize (via spinor approach) the subspaces in the spinor bundle which induce the same group isomorphic to $SU(2)$. Moreover, we show how to induce quaternionic structure on the contact distribution of considered $SU(2)$--structure. We conclude with the invariance of certain components of the covariant derivative $\nabla\varphi$, where $\varphi$ is any spinor field defining considered $SU(2)$--structure. This shows, what expected, that (at least some of) the intrinsic torsion modules, can be derived invariantly with the spinorial approach.     
\end{abstract}

\maketitle   

\section*{Introduction}

The holonomy plays important role in Riemannian geometry. It measures the behavior of the parallel displacement with respect to the Levi-Civita connection. The celebrated theorem by Berger states that the list of possible (restricted) holonomy groups is limited to few cases. In the list there are geometries, called {\it exceptional}, which appear only in certain dimensions:
\begin{itemize}
\item[-] $G_2$ in dimension $7$.
\item[-] ${\rm Spin}(7)$ in dimension $8$,
\end{itemize}
There are two additional geometries, $SU(3)$ in dimension $6$ and $SU(2)$ in dimension $5$ which are also called exceptional (despite the fact that they fall into general category of $SU(n)$--structures in dimension $n$ or $n+1$ as codimension one distribution). This is due to the fact that each of these geometries in dimension $k$ induces appropriate geometry in dimension $k+1$.

In the very interesting articles \cite{ACFH, MM, BMM} the authors study these exceptional geometries from the spinorial point of view. In fact, it can be shown that the unit spinor field induces an exceptional geometry. However, in the $SU(2)$ case the choice of the unit spinor is not unique. 

In this article, we study $SU(2)$--geometry form the perspective of spinors focusing on the invariant approach, i.e., independent on the choice of defining spinor. We concentrate on the action of vectors on spinors. Let us be more precise. Consider a spinor representation $\rho:{\rm Spin}(5)\to {\rm End}(\Delta)$, where $\Delta=\mathbb{C}^4$. Let $\varphi_0\in \Delta$ be a unit spinor (defining a group $SU(2)$). Then an $SU(2)$--structure on a spin $5$--dimensional manifold $M$ is a subbundle $P$ in the bundle ${\rm Spin}(M)$ with the structure group $SU(2)$ or, equivalently, the choice of the unit spinor field $\varphi$ and a subbundle $P$ of all frames $u$ such that $\varphi=[u,\varphi_0]$. 

We show that there is a different choice of the spinor in $\Delta$ defining the same $SU(2)$ \cite{IA}. In fact, we show that the subspace $V\subset \Delta$ of spinors defining the same group $SU(2)$ is of real dimension four. We characterize these spaces from two perspective: complex and quaternionic. The quaternionic approach seems to be well known to the experts in the field, whereas the complex approach is probably new. 

The other case is concerned with the characterization of the intrinsic torsion modules. It is well known \cite{CS} that intrinsic torsion is characterized by the covariant derivative $\nabla\varphi$, where $\varphi$ is a defining spinor. We study invariance of this approach. We show that for a different spinor $\psi$ defining the same $SU(2)$--structure a certain decomposition of $\nabla\psi$ induces the same components. To derive such invariance, we slightly modify used quaternionic structure on $\Delta$ and apply the action of two--forms on spinors.

We begin, in the first section, by algebraic studies of spinors defining $SU(2)$, or equivalently, its Lie algebra $\alg{su}(2)$. Majority of results in this section is well known, however it is hard to find appropriate citations. We give "complex" and "quaternionic" characterization of subspaces of spinors defining the same Lie algebra $\alg{su}(2)$. Moreover, we deal with the correspondence between complex structures on spinor space $\Delta$ and the associated four dinemsional space $D$ of vectors acting on spinors. We show nonexistence of certain correspondence. On the other hand, we show how in a canonical way obtain a quatrnionic structure on $D$ by the invariant spinorial approach. It is interesting, that the map which assigns a complex structure (from the quaternionic structure) to a unit spinor is, in fact, the Hopf fibration. 

In the second section, we show how algebraic approach developed in the first section induces a $SU(2)$--structure on a $5$--dimensional spin manifold. Moreover, we show relations with the approaches from \cite{CS} and \cite{BMM}.

In the final -- third -- section, we show that with the slight modification the spinorial approach developed in \cite{BMM} is invariant, i.e., independent on the choice of a defining spinor.  

\begin{ackn}
I would like to thank Professor Ilka Agricola for pointing that there is no uniqueness of choice of the spinor defining a subgroup $SU(2)\subset{\rm Spin}(5)$ \cite{IA}. This was the starting point of this article. 
\end{ackn}

\section{Decomposition of the space of spinors}

\subsection{Spin representation}
Consider a real Clifford algebra ${\rm Cl}_5$. Then the irreducible representation $\Delta$ of ${\rm Cl}_5$ is complex, $\Delta=\mathbb{C}^4$. It can be given by the following action of vectors $e_i\in\mathbb{R}^5\subset{\rm Cl}_5$ \cite{IA0}:
\begin{align*}
&e_1=\left(\begin{array}{cccc} 0 & 0 & 0 & i \\ 0 & 0 & i & 0 \\ 0 & i & 0 & 0 \\ i & 0 & 0 & 0
\end{array}\right),
&&e_2=\left(\begin{array}{cccc} 0 & 0 & 0 & -1 \\ 0 & 0 & 1 & 0 \\ 0 & -1 & 0 & 0 \\ 1 & 0 & 0 & 0
\end{array}\right),
&&e_3=\left(\begin{array}{cccc} 0 & 0 & -i & 0 \\ 0 & 0 & 0 & i \\ -i & 0 & 0 & 0 \\ 0 & i & 0 & 0
\end{array}\right),\\
&e_4=\left(\begin{array}{cccc} 0 & 0 & 1 & 0 \\ 0 & 0 & 0 & 1 \\ -1 & 0 & 0 & 0 \\ 0 & -1 & 0 & 0
\end{array}\right),
&&e_5=\left(\begin{array}{cccc} i & 0 & 0 & 0 \\ 0 & i & 0 & 0 \\ 0 & 0 & -i & 0 \\ 0 & 0 & 0 & -i
\end{array}\right).
\end{align*}
There is a Hermitian product $(\cdot,\cdot)$ in $\Delta$ such that the spinor representation is unitary and Clifford product by vectors is skew--symmetric. Denote by $\skal{\cdot}{\cdot}$ an inner product which is a real part of $(\cdot,\cdot)$. Fix a (unit) spinor $\varphi\in \Delta$ and define
\begin{equation*}
W_{\varphi}=\{x\cdot\varphi\mid x\in\mathbb{R}^5\}.
\end{equation*}
The action $\mathbb{R}^5\ni x\mapsto x\cdot\varphi\in W_{\varphi}$ is an isomorphism (see, for example, the proof of Lemma \ref{lem:repVphi} below), i.e., $\dim W_{\varphi}=5$. 

We begin with the first well--known easy observation.
\begin{lem}\label{lem:repVphi}
There is the unique vector $y=y_{\varphi}$ such that $y\cdot\varphi=i\varphi$ and the unique complex $2$--dimensional subspace $V_{\varphi}$ in $W_{\varphi}$. They satisfy
\begin{equation*}
W_{\varphi}=V_{\varphi}\oplus\lin{i\varphi}.
\end{equation*}
Moreover, there is a real $4$--dimensional subspace $D_{\varphi}\subset\mathbb{R}^5$ such $V_{\varphi}=D_{\varphi}\cdot\varphi$ and
\begin{equation*}
\Delta=V_{\varphi}\oplus V_{\varphi}^{\bot},
\end{equation*}
where for any $\psi\in V_{\varphi}^{\bot}$ we have
\begin{equation*}
y\cdot\psi=-i\psi.
\end{equation*}
\end{lem}
\begin{proof}
Writing $\varphi=(\varphi^1,\varphi^2,\varphi^3,\varphi^4)\in\mathbb{C}^4$ the action $R_{\varphi}:\mathbb{R}^5\to\mathbb{C}^4$, $R_{\varphi}(x)=x\cdot\varphi$ is represented by the matrix
\begin{equation*}
R_{\varphi}=\left(\begin{array}{ccccc}
i\varphi^4 & \varphi^4 & -i\varphi^3 & \varphi^3 & i\varphi^1\\
i\varphi^3 & \varphi^3 & i\varphi^4 & \varphi^4 & i\varphi^2\\
i\varphi^2 & -\varphi^2 & -i\varphi^1 & -\varphi^1 & -i\varphi^3\\
i\varphi^1 & \varphi^1 & i\varphi^2 & -\varphi^2 & -i\varphi^4
\end{array}\right).
\end{equation*}
It can be checked that the rank of $R_{\varphi}$ is $5$ as a real $5\times 8$ matrix and is equal to the rank of extended block matrix $(R_{\varphi}\, i\varphi^{\top})$. Moreover, if $x$ is a solution to $R_{\varphi}(x)=i\varphi$, then
\begin{equation*}
-|x|^2\varphi=x\cdot x\cdot\varphi=x\cdot(i\varphi)=-\varphi.
\end{equation*}
It implies $|x|=1$. Thus, there is only one solution to the equation $R_{\varphi}(x)=i\varphi$ and this solution has unit norm. Denote it by $y=y_{\varphi}$ and let the action of a vector $y=\sum_i y^ie_i\in\mathbb{R}^5$ on $\Delta$ be denoted by $L_y$. Then $L_y$ is represented by the matrix
\begin{equation*}
\left(\begin{array}{cccc} 
iy^5 & 0 & -iy^3+y^4 & iy^1-y^2 \\
0 & iy^5 & iy^1+y^2 & iy^3+y^4 \\
-iy^3-y^4 & iy^1-y^2 & -iy^5 & 0 \\
iy^1+y^2 & iy^3-y^4 & 0 & -iy^5
\end{array}\right).
\end{equation*}
It is not hard to check that there exists $\tilde{\varphi}$ orthogonal and $\mathbb{C}$--linearly independent with $\varphi$ such that $y\cdot\tilde{\varphi}=i\tilde{\varphi}$. 

We will show that $V_{\varphi}$ equals $\cgtwo{\varphi}{\tilde{\varphi}}^{\bot}$. For $x$ orthogonal to $y$, we have
\begin{equation*}
(x\cdot\varphi,\psi)=(y\cdot x\cdot\varphi,y\cdot \psi)=-(x\cdot y\cdot\varphi,y\cdot\psi)=-(x\cdot\varphi,\psi),\quad \psi\in\{\varphi,\tilde{\varphi}\}.
\end{equation*}
This implies $(x\cdot\varphi,\tilde{\varphi})=0$ and $(x\cdot\varphi,\varphi)=0$ for any $x$ orthogonal to $y$. Put
\begin{equation}\label{eq:Vphi}
V_{\varphi}=\{x\cdot\varphi\mid \skal{x}{y}=0\}.
\end{equation} 
By above $V_{\varphi}$ is orthogonal with respect to $(\cdot,\cdot)$ to the complex space spanned by $\varphi$ and $\tilde{\varphi}$. Thus, $V_{\varphi}$ is complex. By dimensional reasons $V_{\varphi}$ is the unique complex $2$--dimensional subspace in $W_{\varphi}$. Moreover, $D_{\varphi}=\{x\in\mathbb{R}^5\mid \skal{x}{y}=0\}$.
\end{proof}

\subsection{Action on $2$--forms}
We define the action of skew--forms on $\Delta$ as usual
\begin{equation*}
\left(\sum_{i_1<\ldots<i_k}\alpha_{i_1\ldots i_k}e^{i_1}\wedge\ldots\wedge e^{i_k}\right)\cdot\varphi=\sum_{i_1<\ldots<i_k}\alpha_{i_1\ldots i_k} e_1\cdot\ldots e_k\cdot\varphi.
\end{equation*} 
Denote by $\alg{su}(2)_{\varphi}$ the anihilator of this action on two--forms (for given $\varphi$). It is well known that $\alg{su}(2)_{\varphi}$ is isomorphic to the Lie algebra $\alg{su}(2)$. Moreover let $\mathbb{R}^4_{\varphi}$ be the subspace of $\alg{so}(5)$ of $2$--forms $\omega$ such that $\omega\wedge y^{\flat}=0$. In other words $\omega\in\mathbb{R}^4_{\varphi}$ if $\omega=\alpha\wedge y^{\flat}$ for some $1$--form on $D_{\varphi}$. The action of such forms equals
\begin{equation*}
(\alpha\wedge y^{\flat})\cdot\varphi=\alpha\cdot(i\varphi)=i\alpha^{\sharp}\cdot \varphi\in V_{\varphi}. 
\end{equation*}
Since $\dim\alg{su}(2)_{\varphi}=3$, it follows that the subspace $\alg{so}(5)\cdot\varphi\subset \Delta$ is $7$--dimensional and, clearly, orthogonal to $\varphi$. Hence, $\alg{so}(5)\cdot\varphi=\lin{\varphi}^{\bot}$.

\subsection{Complex approach}
\begin{defin}
We say that a complex $2$--dimensional subspace $V$ in $\Delta$ is {\it admissible} if for any $\varphi\in V^{\bot}$ we have $V\subset W_{\varphi}$.
\end{defin}

\begin{lem}
Fix (unit) $\varphi\in \Delta$. Then $V_{\varphi}$ is admissible.
\end{lem}
\begin{proof}
By Lemma \ref{lem:repVphi} $V_{\varphi}=V_{\tilde{\varphi}}$, where $\tilde{\varphi}$ is as in the proof of Lemma \ref{lem:repVphi}. Take any linear combination $\psi$ of $\varphi$ and $\tilde{\varphi}$. Then, $y_{\psi}=y_{\varphi}$, hence $D_{\psi}=D_{\varphi}$. Moreover, for $x$ orthogonal to $y_{\psi}$ we have
\begin{equation*}
x\cdot\psi=a x\cdot\varphi+b x\cdot\tilde{\varphi}\in V_{\varphi}
\end{equation*}
for some $a,b\in\mathbb{C}$. Thus $V_{\psi}=V_{\varphi}$.
\end{proof}

\begin{lem}\label{lem:admissible}
Let $V$ be an admissible subspace. Then 
\begin{enumerate}
\item there is unique unit vector $y\in\mathbb{R}^5$ such that $y\cdot\varphi=i\varphi$ for any $\varphi\in V^{\bot}$,
\item $D_{\varphi}$ coincide for all $\varphi\in V^{\bot}$,
\item $V=V_{\varphi}$ for any $\varphi\in V^{\bot}$.
\end{enumerate}
\end{lem} 
\begin{proof}
Fix $\varphi\in V^{\bot}$. By Lemma \ref{lem:repVphi} there is unique $y$ such that $y\cdot \varphi=i\varphi$ and $W_{\varphi}=V_{\varphi}\oplus\lin{i\varphi}$, where $V_{\varphi}=\{x\cdot\varphi\mid \skal{x}{y}=0\}$. Thus $D_{\varphi}=\lin{y}^{\bot}$. Moreover, there is $\tilde{\varphi}$ which is $\mathbb{C}$--linearly independent with $\varphi$ and such that $y\cdot\tilde{\varphi}=i\tilde{\varphi}$ and $\Delta=V_{\varphi}\oplus\cgtwo{\varphi}{\tilde{\varphi}}$. Since $V$ is maximal complex in $W_{\varphi}$, by admissibility we have $V_{\varphi}=V$. This proves the third condition. 

Now, take any $\psi\in V^{\bot}$. Then, $\psi=a\varphi+b\tilde{\varphi}$ for some $a,b\in\mathbb{C}$. Thus $y\cdot\psi=i\psi$, which implies $D_{\psi}=D_{\varphi}$, what proves the first and the second condition.
\end{proof}

\begin{thm}\label{thm:su2equivalence}
Assume $V$ is admissible. Then $\alg{su}(2)_{\varphi}$ coincide for all $\varphi\in V^{\bot}$. Conversely, for the maximal space $V^{\bot}$ such that all $\alg{su}(2)_{\varphi}$ coincide for $\varphi\in V^{\bot}$ the orthogonal complement $V$ is admissible.
\end{thm}
\begin{proof}
Assume $V$ is admissible. Take an orthonormal $\mathbb{C}$--basis $\{\varphi, \tilde{\varphi}\}$ of $V^{\bot}$. By Lemma \ref{lem:admissible} $D_{\varphi}=D_{\tilde{\varphi}}$. We will write just $D$. Choose an orthonormal basis $(e_j)$ of $D$ and let $(u_j)$ be a basis of $D$ such that $e_j\cdot\varphi=u_j\cdot\tilde{\varphi}$. Then
\begin{equation*}
\skal{u_j}{u_k}=\skal{u_j\cdot\tilde{\varphi}}{u_k\cdot\tilde{\varphi}}=\skal{e_j\cdot\varphi}{e_k\cdot\varphi}=\skal{e_j}{e_k}.
\end{equation*}
Hence $(u_j)$ is also orthonormal. Moreover, 
\begin{equation*}
0=\skal{\varphi}{\tilde{\varphi}}=\skal{e_j\cdot\varphi}{e_j\cdot\tilde{\varphi}}=\skal{u_j\cdot\tilde{\varphi}}{e_j\cdot\tilde{\varphi}}=\skal{e_j}{u_j}.
\end{equation*}
Fixing $e_1$ and corresponding $u_1$, we may take $e_2=u_1$. Hence, $u_1\cdot\varphi=u_2\cdot\tilde{\varphi}$. We have
\begin{equation*}
\skal{e_1}{u_2}=\skal{e_1\cdot\varphi}{u_2\cdot\varphi}=\skal{u_1\cdot\tilde{\varphi}}{u_2\cdot\varphi}=-\skal{u_2\cdot\tilde{\varphi}}{u_1\cdot\varphi}=-\skal{u_1\cdot\varphi}{u_1\cdot\varphi}=-1.
\end{equation*}
Since both $e_1$ and $u_2$ are unit, it follows that $u_2=-e_1$. In particular, ${\rm span}\{e_1,e_2\}={\rm span}\{u_1,u_2\}$. Analogously, we set $e_4=u_3$ and we obtain $u_4=-e_3$. Consider the following $2$--forms
\begin{equation}\label{eq:su2basis}
\omega_0=e_1\wedge u_1-e_3\wedge u_3,\quad
\omega_1=e_1\wedge e_3+u_1\wedge u_3,\quad
\omega_2=e_1\wedge u_3-u_1\wedge e_3.
\end{equation}
It is not hard to check that $\omega_j\cdot\varphi=\omega_j\cdot\tilde{\varphi}=0$ and $\omega_j$ are linearly independent. In particular $\omega_j\cdot\psi=0$ for any $\mathbb{C}$--linear combination of $\varphi$ and $\tilde{\varphi}$. Hence, all $\alg{su}(2)_{\varphi}$ for $\varphi\in V^{\bot}$ coincide.

Conversely, let $\omega_0,\omega_1,\omega_2$ be $2$--forms in $\mathbb{R}^5$ defining a Lie algebra $\alg{g}$ isomorphic to $\alg{su}(2)$ and let $V^{\bot}$ be a subspace of these $\varphi\in\Delta$ such that $\omega_j\cdot\varphi=0$. Clearly, $V^{\bot}$ is complex. 

Fix $\varphi\in V^{\bot}$. By Lemma \ref{lem:repVphi}, $V_{\varphi}$ is complex $2$--dimensional and orthogonal to $\varphi$. By Lemma \ref{lem:admissible}, $V_{\varphi}$ is admissible. Since $\tilde{\varphi}\in V_{\varphi}^{\bot}$, where $\tilde{\varphi}$ is as above, by the first part $\alg{su}(2)_{\tilde{\varphi}}=\alg{su}(2)_{\varphi}=\alg{g}$, i.e., $\omega_j\cdot\tilde{\varphi}=0$. Thus $\tilde{\varphi}\in V^{\bot}$. We have shown that $V_{\varphi}^{\bot}\subset V^{\bot}$. In other words, $V\subset V_{\varphi}\subset W_{\varphi}$. It suffices to show that $\dim_{\mathbb{C}}V=\dim_{\mathbb{C}}V^{\bot}=2$. Suppose $\dim_{\mathbb{C}}V=1$ and let $u$ be a unit vector in $\mathbb{R}^5$ orthogonal to $y_{\varphi}$ and such that $\psi=u\cdot\varphi$ is orthogonal to $V$. Then $\psi\in V^{\bot}$, hence $\alg{su}(2)_{\psi}=\alg{g}$. Therefore
\begin{equation}\label{eq:alggso3}
0=u\cdot\omega_j\cdot\varphi=2(u\lrcorner\omega_j)\cdot\varphi+\omega_j\cdot\psi=2(u\lrcorner\omega_j)\cdot\varphi.
\end{equation}
For $x\in D_{\varphi}$, by the fact that $V_{\varphi}$ is a complex subspace, we have
\begin{equation*}
(x\wedge y)\cdot\varphi=ix\cdot\varphi\in V_{\varphi}.
\end{equation*}
This implies $\mathbb{R}^4_{\varphi}\cdot\varphi=V_{\varphi}$. Since $\alg{so}(5)\cdot\varphi=\lin{\varphi}^{\bot}$, it follows that $\alg{g}\subset \alg{so}(D_{\varphi})$. By \eqref{eq:alggso3}, we see that in fact $\alg{g}\subset\alg{so}(3)$, where $\alg{so}(3)$ is taken with respect to the $3$--dimensional subspace of $D_{\varphi}$ orthogonal to $u\in D_{\varphi}$. Thus $\alg{g}=\alg{so}(3)$. This is impossible, since $\alg{so}(3)$ contains pure elements $\omega=\alpha\wedge\beta$ and $\omega\cdot\varphi$ cannot vanish. Finally, $\dim_{\mathbb{C}}V=2$.
\end{proof}

Let us now describe the decomposition of $\alg{so}(5)$ into irreducible $\alg{su}(2)$--modules. Fix an admissible space $V^-$ and denote by $\alg{su}(2)_-$ the corresponding, by above theorem, Lie algebra isomorphic to $\alg{su}(2)$. Let $V^+=(V^-)^{\bot}$. Later, we will also use the following convention: if $V$ is an admissible space, then $\alg{su}(2)_V$ denotes the Lie algebra induced by spinors $\varphi\in V^{\bot}$, i.e., $\alg{su}(2)_V=\alg{su}(2)_{\varphi}$ for any $\varphi\in V^{\bot}$. 

\begin{lem}\label{lem:adjointsu2}
The subspace $V^+$ is admissible. Denoting by $\alg{su}(2)_+$ the corresponding Lie algebra isomorphic to $\alg{su}(2)$ we have
\begin{equation*}
\Delta=V^-\oplus V^+,\quad \alg{so}(5)=\alg{su}(2)_-\oplus\alg{su}(2)_+\oplus\mathbb{R}^4_{\varphi},
\end{equation*}
where $\varphi\in V^-\cup V^+$.
\end{lem}
\begin{proof}
Let $\psi\in V^-$. Choose any $\varphi\in V^+$. By Lemma \ref{lem:admissible} $V^-=V_{\varphi}$, hence, there is $x_0\in\mathbb{R}^5$ such that $\psi=x_0\cdot\varphi$. Therefore, $x_0\cdot\psi=-|x_0|^2\varphi$, i.e., $\varphi\in W_{\psi}$. We have shown that $V^+\subset W_{\psi}$, thus $V^+$ is admissible. 

By Lemma \ref{lem:repVphi}, $x_0$ is orthogonal to $y$, where $y\cdot\varphi=i\varphi$. Thus
\begin{equation*}
y\cdot\psi=y\cdot x_0\cdot\varphi=-x_0\cdot y\cdot\varphi=-i(x_0\cdot\varphi)=-i\psi.
\end{equation*} 
Hence $y_{\psi}=-y$. In particular, $D_{\psi}=D_{\varphi}$ for any $\varphi\in V^+$. By the proof of Theorem \ref{thm:su2equivalence} we have $\alg{su}(2)_{\psi}\subset \alg{so}(D_{\psi})$. For $\omega\in \alg{su}(2)_{\varphi}$ we have
\begin{equation*}
\omega\cdot\psi=\omega\cdot x_0\cdot\varphi=2(x_0\lrcorner\omega)\cdot\varphi+x_0\cdot\omega\cdot\varphi=2(x_0\lrcorner\omega)\cdot\varphi.
\end{equation*}
The right hand side vanishes only if $x_0\lrcorner\omega=0$. Choose a basis $(e_j)$ in $D_{\varphi}$ such that $e_1=\frac{1}{|x_0|}x_0$ and a basis $(u_j)$ as in the proof of Theorem \ref{thm:su2equivalence}. Then $\omega$ is a linear combination of $\omega_0,\omega_1,\omega_2$ given by \eqref{eq:su2basis}, say $\omega=\sum_j a_j\omega_j$. If $\omega\neq 0$, then
\begin{equation*}
x_0\lrcorner\omega=a_0u_1+a_1e_3+a_2u_3\neq 0.
\end{equation*}
Thus $\omega\cdot\varphi\neq 0$. Therefore, $\alg{su}(2)_{\psi}$ is transversal to $\alg{su}(2)_{\varphi}$. This completes the proof.
\end{proof}

\begin{cor}\label{cor:su2actions}
Let $V^-$ be an admissible space. Then the following actions are surjective
\begin{equation*}
\alg{su}(2)_-:V^+\to V^+,\quad \alg{su}(2)_+:V^-\to V^-.
\end{equation*}
\end{cor}
\begin{proof}
Follows from the fact that for a given spinor $\varphi\in V^+$ we have $\alg{su}(2)_+\cdot\varphi=V^+\cap\lin{\varphi}^{\bot}$.
\end{proof}

Notice that by Theorem \ref{thm:su2equivalence} the subspace $\mathbb{R}^4_{\varphi}\subset \alg{so}(5)$ is independent of the choice of $\varphi\in (V^-)^{\bot}$. Hence we may denote it by $\mathbb{R}^4_-$.

\begin{rem}\label{rem:spinrep}
From the considerations above we have useful observation concerning the Clifford action. Namely, let $V^-$ be admissible and let $\varphi\in (V^-)^{\bot}=V^+$ be unit. Then the right multiplication $R_{\varphi}:\alg{so}(5)\to\Delta$ by $\varphi$ satisfies the following restrictions
\begin{align*}
& R_{\varphi}:\mathbb{R}^4_-\to V^+,\\
& R_{\varphi}:\alg{su}(2)_+\to \lin{\varphi}^{\bot}\cap V^-
\end{align*}
are isomorphisms. 
\end{rem}

\subsection{A fundamental example}

Denote by $s_1,\ldots,s_4$ the canonical $\mathbb{C}$--basis in $\Delta=\mathbb{C}^4$. Fix a spinor $\varphi=s_1$. Then
\begin{equation*}
e_1\varphi=is_4,\quad e_2\varphi=s_4,\quad e_3\varphi=-is_3,\quad e_4\varphi=-s_3,\quad e_5\varphi=is_1.
\end{equation*}
Hence $V^-=\langle s_3,s_4\rangle_{\mathbb{C}}$ and $V^+=\langle s_1,s_2\rangle_{\mathbb{C}}$. Moreover,
\begin{align*}
&e_{12}\varphi=is_1,\quad e_{13}\varphi=s_2,\quad e_{14}\varphi=-is_2,\quad e_{15}\varphi=-s_4,\\
&e_{23}\varphi=-is_2,\quad e_{24}\varphi=-s_2,\quad e_{25}\varphi=is_4,\\
&e_{34}\varphi=is_1,\quad e_{35}\varphi=s_3,\\
&e_{45}\varphi=-is_3.
\end{align*}
Here and further, $e_{jk}$ denotes the two--form $e_j\wedge e_k$. Now, it is easy to see that the equation $\omega\cdot\varphi=0$ is satisfied by the following $2$--forms
\begin{equation*}
\tilde{\omega}_1=e_{12}-e_{34},\quad \tilde{\omega}_2=e_{13}+e_{24},\quad \tilde{\omega}_3=e_{14}-e_{23},
\end{equation*}
which define $\alg{su}(2)_-$. Hence, $(\alg{su}(2)_-)^{\bot}\subset\alg{so}(5)$ is spanned by $\alg{su}(2)_+$ generated by elements
\begin{equation*}
\omega_1=e_{12}+e_{34},\quad \omega_2=e_{13}-e_{24},\quad \omega_3=e_{14}+e_{23},
\end{equation*}
and $\mathbb{R}^4_-$ generated by elements
\begin{equation*}
e_{15},\quad e_{25},\quad e_{35},\quad e_{45}.
\end{equation*}

Notice that for the spinor $s_2$ we have
\begin{equation*}
e_1\cdot s_2=is_3,\quad e_2\cdot s_2=-s_3,\quad e_3\cdot s_2=is_4,\quad e_4\cdot s_2=-s_4,\quad e_5\cdot s_2=is_2,
\end{equation*}
which confirms that $V^-$ and $V^+$ are exactly as stated above. Moreover, we have the following relation
\begin{equation}\label{eq:rels1s2}
e_1\cdot s_1=e_3\cdot s_2,\quad e_2\cdot s_1=-e_4\cdot s_2,\quad e_3\cdot s_1=-e_1\cdot s_2,\quad e_4\cdot s_1=e_2\cdot s_2. 
\end{equation}

\subsection{Complex structures - negative results}

We will show nonexistence of complex structures on $\Delta$ satisfying certain relations. 

Firstly, assuming $j:\Delta\to\Delta$ is a complex structure on a real vector space $\Delta$, there is a complex structure $I_{\varphi}$ on $D_{\varphi}$ for fixed $\varphi$ in the orthogonal complement of the admissible space $V_{\varphi}$ (compare \cite{ACFH, BMM}). Namely,
\begin{equation}\label{eq:complexstrDvarphi}
I_{\varphi}(x)\cdot\varphi=j(x\cdot\varphi),\quad x\in D_{\varphi}.
\end{equation}  
The definition of $I_{\varphi}$ depends on $\varphi$. Moreover, in the definition \eqref{eq:complexstrDvarphi} we only need the values of $j$ on $V_{\varphi}$, since for any $x\in D_{\varphi}$ we have $x:V_{\varphi}^{\bot}\to V_{\varphi}$. 

We want to find all complex structures $j$ on $V$ such that the induced complex structure $I_{\varphi}$ is independent of the choice of $\varphi\in V^{\bot}$. In fact, what is stated and proved below, there is no such complex structure. We will prove slightly more.

Let us state nonexistence result in a more general setting. Let $t:V\to V$ be a $\mathbb{R}$--linear map and define the induced linear map $T^{\varphi}:D_{\varphi}\to D_{\varphi}$, for $\varphi\in V^{\bot}$, analogously as in \eqref{eq:complexstrDvarphi}
:
\begin{equation}\label{eq:tinducesT}
T_{\varphi}(x)\cdot\varphi=t(x\cdot \varphi),\quad x\in D_{\varphi}.
\end{equation}

\begin{thm}\label{thm:complexstructurenonexistence}
The only linear map $t:V\to V$ such that the induced linear map $T_{\varphi}$ is independent of the choice of $\varphi\in V^{\bot}$ is a scalar multiple of the identity map.
\end{thm}

\begin{proof}
Without loss of generality, we may take $V$ to be $\cgtwo{s_3}{s_4}$. Assume $t$ induces a map $T$, which does not depend on the choice of the spinor in $V^{\bot}=\cgtwo{s_1}{s_2}$. Then $D={\rm span}_{\mathbb{R}}\{e_1,e_2,e_3,e_4\}$. Let $T(e_1)=(a,b,c,d)$. Then
\begin{align*}
& t(s_3)=-(a,b,c,d)\cdot (is_2), && t(is_3)=(a,b,c,d)\cdot s_2,\\
& t(s_4)=-(a,b,c,d)\cdot (is_1), && t(is_4)=(a,b,c,d)\cdot s_1.
\end{align*}
This implies that the matrix of $t:V\to V$ with respect to the basis $s_3,is_3,s_4,is_4$ is of the form
\begin{equation*}
t=\left(\begin{array}{cccc}
a & -b & -c & -d \\ b & a & d & -c \\ c & -d & a & b \\ d & c & -b & a
\end{array}\right).
\end{equation*}
From this we conclude that $T$ is represented by the same matrix with respect to the basis $e_1,e_2,e_3,e_4$.
 
Let us study independence on $\varphi$. For $x,y\in D_{\varphi}$ and $\varphi,\psi\in V^{\bot}$ such that $x\cdot\varphi=y\cdot\psi$ it must hold
\begin{equation*}
T(x)\cdot \varphi=T(y)\cdot \psi.
\end{equation*}
Substituting relations \eqref{eq:rels1s2} we obtain $b=d=0$. Moreover, for $\varphi\in V^{\bot}$ and $x\in D_{\varphi}$ there is $y$ (depending on $\varphi$) such that $y\cdot\varphi=x\cdot i\varphi$. Considering this condition we conclude, as above, that $c=0$. Hence, $t=a\cdot{\rm id}_V$. Finally, it is clear that such $t$ satisfies assumptions of the theorem.
\end{proof}

\begin{cor}\label{cor:nonexistednce}
There is no complex structure $I$ on $D_{\varphi}$ induced from the complex structure $j$ on $V$ by the formula \eqref{eq:complexstrDvarphi}, which is independent on the choice of $\varphi\in V^{\bot}$.  
\end{cor}

\subsection{Complex structures - positive results}

In this subsection, we give a natural procedure how to define an associated quaternionic structure on an $SU(2)$--structure via spinorial approach. In a previous subsection we have shown that an approach similar to the one considered in \cite{ACFH} (compare \cite{BMM}) is not valid. Nevertheless, these is nice description of complex structures on $D_{\varphi}$ with the spinorial (invariant) approach. The intuition has been already used in the proof of Theorem \ref{thm:complexstructurenonexistence}.

Let $V$ be an admissible space. We know that all $D_{\varphi}$ coincide for $\varphi\in V^{\bot}$. Denote this space by $D$. Fix $\varphi\in V^{\bot}$ and $x\in D$. Then there exists unique element $J^{\varphi}(x)\in D$ such that
\begin{equation*}
J^{\varphi}(x)\cdot\varphi=x\cdot (i\varphi)=i(x\cdot \varphi).
\end{equation*}
The second equality follows from the fact that the action of $x\in D$ on $\Delta$ is $\mathbb{C}$--linear.

Firstly, notice that the map $J^{\varphi}:D\to D$ is a complex structure,
\begin{equation*}
J^{\varphi}(J^{\varphi}(x))\cdot\varphi=i(J^{\varphi}(x)\cdot\varphi)=-x\cdot\varphi.
\end{equation*}
Hence $(J^{\varphi})^2=-{\rm id}_D$. 

Secondly, we have $J^{\lambda\varphi}=J^{\varphi}$ for any complex number $\lambda\neq 0$. Thus we have a correspondence
\begin{equation}\label{eq:complexstrvarphi}
\{\varphi\in V^{\bot}:\, |\varphi|=1\}\mapsto \{\textrm{complex str. in $D$}\},\quad \varphi\mapsto J^{\varphi}.
\end{equation}

\begin{prop}\label{prop:quaternionic}
The correspondence \eqref{eq:complexstrvarphi} is in fact the Hopf fibration $\mathbb{S}^1\to\mathbb{S}^3\to\mathbb{S}^2$. In particular, the image $\{J^{\varphi}:\varphi\in D\}$ is a $2$--sphere which defines a quaternionic structure on $D$.  
\end{prop}
\begin{proof}
Without loss of generality, we may take $V=\cgtwo{s_3}{s_4}$. Then $V^{\bot}=\cgtwo{s_1}{s_2}$. Any spinor $\varphi\in V^{\bot}$ equals $\varphi=as_1+bis_1+cs_2+dis_2$ for some $a,b,c,d\in\mathbb{R}$. Moreover, let $J^{\varphi}(e_1)=\sum_j x_je_j$. Substituting $x$ by $e_1$ in the definition of $J^{\varphi}$ we get
\begin{equation*}
\left(\begin{array}{cccc}
-d & -c & b & -a \\ c & -d & -a & -b \\ -b & a & -d & -c \\ a & b & c & d
\end{array}\right)\left(\begin{array}{c} x_1 \\ x_2 \\ x_3 \\ x_4\end{array}\right)=\left(\begin{array}{c} -c \\ -d \\ -a \\ -b \end{array}\right)
\end{equation*}
This implies $x_1=0$, $x_2=-(a^2+b^2-c^2-d^2)$, $x_3=2(ad-bc)$, $x_4=2(ac+db)$. We proceed in a similar way taking $x=e_2,e_3,e_4$, respectively. Finally, denoting by
\begin{equation*}
\alpha=a^2+b^2-c^2-d^2,\quad \beta=2(ad-bc),\quad \gamma=2(ac+bd),
\end{equation*}
we obtain
\begin{equation*}
J^{\varphi}=\left(\begin{array}{cccc}
0 & \alpha & -\beta & -\gamma \\ -\alpha & 0 & -\gamma & \beta \\ \beta & \gamma & 0 & \alpha \\ \gamma & -\beta & -\alpha & 0
\end{array}\right).
\end{equation*}
Notice that $\alpha^2+\beta^2+\gamma^2=1$ and that the map $(a,b,c,d)\mapsto (\alpha,\beta,\gamma)$ is the Hopf fibration of $\mathbb{S}^3$ onto $\mathbb{S}^2$.

Let us end bu showing that the $2$--sphere of complex structures defines a quaternionic structure on $D$. Denote by $(\alpha,\beta,\gamma)\in\mathbb{S}^2$ the point on the $2$--sphere inducing a complex structure $J=J(\alpha,\beta,\gamma)$ from the image of the map \eqref{eq:complexstrvarphi}. Then,
\begin{equation*}
J(\alpha,\beta,\gamma)J(\tilde{\alpha},\tilde{\beta},\tilde{\gamma})=-J(\tilde{\alpha},\tilde{\beta},\tilde{\gamma})J(\alpha,\beta,\gamma)
\end{equation*}
if and only if the vectors $(\alpha,\beta,\gamma)$ and $(\tilde{\alpha},\tilde{\beta},\tilde{\gamma})$ are orthogonal (with respect to the standard inner product in  $\mathbb{R}^3$). Now, it suffices to take two orthogonal vectors $u,v\in\mathbb{S}^3$ and define
\begin{equation*}
J_1=J(u),\quad J_2=J(v),\quad J_3=J_1J_2.
\end{equation*}
It is easy to see that the triple $(J_1,J_2,J_3)$ is a quaternionic structure.
\end{proof}

\subsection{Quaternionic approach}

Let us rewrite some of the results from the subsection 1.1 with the quaternionic approach. 

Let us begin with the choice of a quaternionic structure on $\Delta$. The choice is not unique. We begin with the one considered in \cite{BMM}. It can be shown \cite{TF} that in $\Delta$ there is a quaternionic structure $i_2:\Delta\to\Delta$ which anticommutes with the multiplication by vectors. Recall that a quaternionic structure $j$ may be seen as an antilinear map such that $j^2=-{\rm id}$. Let $i_1$ be a complex structure on $\Delta=\mathbb{C}^4$ induced by the volume element ${\rm vol}=e_1\cdot e_2\cdot e_3\cdot e_4\cdot e_5$. It is easy to see that ${\rm vol}$ induces the standard complex structure given by the multiplication by $i$, which clearly commutes with multiplication by vectors. Define $i_3=i_1\circ i_2$. $i_3$ anticommutes with the multiplication. Then we have a triple $(i_1,i_2,i_3)$ of complex structures on $\Delta$. Each $i_k$ is an isometry \cite{BMM}.

We will need the following useful fact.
\begin{lem}[\cite{BMM}]\label{lem:bmm}
Fix a unit spinor $\varphi\in\Delta$. Then a subspace $V^{\bot}$ generated by $\varphi,i_1\varphi,i_2\varphi,i_3\varphi$ and its orthogonal complement $V$ are $i_k$--invariant, $k=1,2,3$. Moreover, the subspace $D\subset\mathbb{R}^5$ such that $D\cdot\varphi=V$ inherits a quaternionic structure induced by complex structures $I_k$, $k=1,2,3$, defined by $I_k(x)\cdot\varphi=i_k(x\cdot\varphi)$. In particluar, each $i_k$ leaves the decomposition $\Delta=V\oplus V^{\bot}$ invariant.
\end{lem}

This allows to prove the main result of this subsection.

\begin{thm}\label{thm:admissiblequaternionic}
A real $4$--dimensional subspace $V\subset \Delta$ is admissible if and only if it is a quaternionic subspace with respect to $(i_1,i_2,i_3)$. 
\end{thm}
\begin{proof} 
Firstly, assume $V$ is admissible. Choose any $\varphi\in V^{\bot}$. Then $V=V_{\varphi}$. By Lemma \ref{lem:bmm}, it suffices to show that $V^{\bot}=\langle \varphi,i_1\varphi,i_2\varphi,i_3\varphi \rangle$. Since $i_1$ is the multiplication by $i$ and $i_3\varphi=i i_2\varphi$ we need to show that $i_2\varphi$ is orthogonal to $V$. Any element in $V$ is of the form $x\cdot\varphi$. We may assume $x$ is unit. Hence
\begin{equation*}
\skal{i_2\varphi}{x\cdot\varphi}=-\skal{\varphi}{i_2(x\cdot\varphi)}=\skal{\varphi}{x\cdot i_2\varphi}=-\skal{x\cdot\varphi}{i_2\varphi}.
\end{equation*}
Thus $\skal{i_2\varphi}{x\cdot\varphi}=0$.

Conversely, assume $V$ is quaternionic. Thus its orthogonal complement $V^{\bot}$ is also quaternonic. Take $\varphi\in V^{\bot}$. By Lemma \ref{lem:admissible} it suffices to show that $V=V_{\varphi}$. 

We have $V^{\bot}=\langle \varphi,i_1\varphi,i_2\varphi,i_3\varphi\rangle$. By Lemma \ref{lem:bmm} there is a subspace $D\subset\mathbb{R}^5$ such that $D\cdot\varphi=V$. Since $V$ is complex, by Lemma \ref{lem:repVphi}, $V=V_{\varphi}$.
\end{proof}

For out approach, to make description, at least partially, invariant, we need to modify a quaternionic structure $(i_1,i_2,i_3)$ a little bit. The modification depends on the choice of the admissible space $V$. Let $j_1=i_1$ and we define $j_2$ as follows
\begin{equation*}
j_2=i_2\quad\textrm{on $V$},\quad j_2=-i_2\quad\textrm{on $V^{\bot}$}.
\end{equation*}  
Finally, let $j_3=j_1 j_2$. Then $(j_1,j_2,j_3)$ equals $(i_1,i_2,i_3)$ on $V$ and $(i_1,-i_2,-i_3)$ on $V^{\bot}$. The triple $(j_1,j_2,j_3)$ is in fact a quaternionic structure (since $V$ and $V^{\bot}$ are invariant with respect to $(i_1,i_2,i_3)$). Moreover,
\begin{equation*}
j_k(x\cdot \varphi)=x\cdot j_k(\varphi),\quad \varphi\in V^+\cup V^-.
\end{equation*}
Except for $j_1$, the complex structures $j_2$ and $j_3$ do not commute in general with the multiplication by vectors.

Now, we move to description of a quaternionic structure on $D$. As discussed in the previous subsection take
\begin{equation*}
J_1=J(1,0,0),\quad J_2=J(0,1,0),\quad J_3=(0,0,-1).
\end{equation*}
Then $J_3=J_1J_2$ and $J_kJ_l=-J_lJ_k$ for distinct $k,l$. There are unit spinors (not unique) $\varphi_1, \varphi_2, \varphi_3\in V^{\bot}$ such that
\begin{equation*}
x\cdot i\varphi_k=J_k(x)\cdot \varphi_k.
\end{equation*}
Moreover, define three $2$--forms $\omega_k$ by
\begin{equation*}
\omega_k(x,y)=\skal{J_k(x)}{y},\quad x,y\in D.
\end{equation*}
Then
\begin{equation*}
x\cdot \omega_k\cdot \varphi_k=\omega_k\cdot x\cdot \varphi_k+2(x\lrcorner\omega_k)\cdot \varphi_k=2(x\lrcorner\omega_k)\cdot \varphi_k=2J_k(x)\cdot \varphi_k=x\cdot(i\varphi_k),
\end{equation*}
which implies
\begin{equation*}
\omega_k\cdot \varphi_k=2i\varphi_k.
\end{equation*}
This relation shows that $\omega_k$ belongs to the Lie algebra dual to $\alg{su}(2)$.

\begin{rem}\label{rem:spinorsinVbot}
If $V=\cgtwo{s_3}{s_4}$, then we may take, for example,
\begin{equation*}
\varphi_1=s_1,\quad \varphi_2=\frac{1}{\sqrt{2}}(s_1+is_2),\quad \varphi_3=\frac{1}{\sqrt{2}}(s_1-s_2).
\end{equation*}
Moreover, the $2$--forms $\omega_k$ are 
\begin{equation*}
\omega_1=e_{12}+e_{34},\quad \omega_2=-e_{13}+e_{24},\quad \omega_3=e_{14}+e_{23}
\end{equation*}
and the corresponding complex structures $J_1,J_2,J_3$ are given by the following matrices
\begin{equation*}
J_1=\left(\begin{array}{cccc}
0 & -1 & 0 & 0 \\
-1 & 0 & 0 & 0 \\
0 & 0 & 0 & 1 \\
0 & 0 & -1 & 0
\end{array}\right),\,
J_2=\left(\begin{array}{cccc}
0 & 0 & -1 & 0 \\
0 & 0 & 0 & 1 \\
1 & 0 & 0 & 0 \\
0 & -1 & 0 & 0
\end{array}\right),\,
J_3=\left(\begin{array}{cccc}
0 & 0 & 0 & -1 \\
0 & 0 & -1 & 0 \\
0 & 1 & 0 & 0 \\
1 & 0 & 0 & 0
\end{array}\right).
\end{equation*}

We would like to add that the construction of a quaternionic structure from the given data was studied in \cite{BV} where the authors use the approach from \cite{CS}. The obtained (almost) complex structures agree with our approach. 
\end{rem}

\subsection{Conjugacy classes}

In this subsection we find the condition for admissible spaces, to induce conjugate Lie algebras (equivalently, groups) isomorphic to $\alg{su}(2)$. More precisely, we deal with the problem when $\alg{su}(2)_V$ and $\alg{su}(2)_{V'}$ are conjugate, where $V,V'$ are admissible. Recall that $\alg{su}(2)_V$ is a Lie algebra of all $\omega$ such that $\omega\cdot\varphi=0$ for any $\varphi\in V^{\bot}$.

\begin{lem}\label{lem:twoadmissible}
Assume $V\subset\Delta$ is an admissible space and let $g\in {\rm Spin}(5)$. Then $gV$ is also admissible.
\end{lem} 
\begin{proof}
Since $g$ acts as a complex linear map, it follows that the space $gV$ is complex. Moreover, let $\psi\in(gV)^{\bot}$ and let $\varphi\in V$. By the invariance of the Hermitian product, we see that $g^{-1}\psi\in V{^\bot}$. Since $V$ is admissible, it follows that $V\subset W_{g^{-1}\psi}$. Thus $\varphi=x\cdot(g^{-1}\psi)$ for some $x\in\mathbb{R}^5$. Hence
\begin{equation*}
g\varphi=({\rm Ad}(g)x)\cdot\psi.
\end{equation*}
Since ${\rm Ad}(g)x$ is a vector in $\mathbb{R}^5$, we have $g\varphi\in W_{\psi}$. This proves admissibility of $gV$.
\end{proof}

\begin{lem}\label{lem:Spinactionadmissible}
The isotropy group of the fixed element of the action of ${\rm Spin}(5)$ on admissible subspaces is isomorphic to ${\rm Spin}(4)$.
\end{lem}
\begin{proof}
It suffices to take $\varphi=s_1$. Then $V=\cgtwo{s_3}{s_4}$. It is easy to see that $gV=V$ if and only if $g\in{\rm Spin}(4)$ where we consider the spin group with respect to the first component in the decomposition $\mathbb{R}^5=\mathbb{R}^4\oplus\mathbb{R}$.
\end{proof}

\begin{prop}\label{prop:conjugateSU}
Let $g\in{\rm Spin}(5)$. We have 
\begin{equation*}
{\rm Ad}(g)\alg{su}(2)_V=\alg{su}(2)_{g^{-1}V}.
\end{equation*}
\end{prop}
\begin{proof}
Let $\omega\in\alg{su}(2)_V$, i.e., $\omega\cdot\varphi=0$ for all $\varphi\in V^{\bot}$. Thus $0=({\rm Ad}(g)\omega)\cdot (g\varphi)$, which implies ${\rm Ad}(g)\alg{su}(2)_V=\alg{su}(2)_{g\varphi}$. Since $g\varphi\in (gV)^{\bot}$, by admissibility of $gV$, proposition follows. 
\end{proof}

\begin{cor}\label{cor:conjugateSU}
Let $g\in{\rm Spin}(5)$. Then ${\rm Ad}(g)\alg{su}(2)_V=\alg{su}(2)_V$, for $V$ admissible, if and only if $gV=V$. In particular, the stabilizer of $\alg{su}(2)_V$ with respect to the adjoint action of ${\rm Spin}(5)$ is isomorphic to ${\rm Spin}(4)$.
\end{cor}

Let us compare above considerations with the quaternionic approach. Choose a quaternionic structure $(j_1,j_2,j_3)$ on $\Delta$ in a way such that each $j_k$ commutes or anticommutes with the multiplication by vectors (see subsection 1.7). Consider an action of $\mathbb{H}$ in $\Delta$ in a natural way: for $a=a_0+a_1i+a_2j+a_3k\in\mathbb{H}$ and $\varphi\in\Delta$ let
\begin{equation*}
a\cdot \varphi=a_0\varphi+a_1j_1\varphi+a_2j_2\varphi+a_3j_3\varphi.
\end{equation*}
This action commutes with the action of ${\rm Spin}(5)$. Moreover, by Theorem \ref{thm:admissiblequaternionic} each admissible space is of the form $V_{\varphi}=\{a\cdot\varphi \mid a\in\mathbb{H}\}$ for each $\varphi\in \Delta$. These arguments give another proof of Lemma \ref{lem:Spinactionadmissible}.

The quotient space $\Delta/\mathbb{H}$ of this action, which is isomorphic to $\mathbb{R}^4$, is a space of all admissible spaces. In addition, for fixed $\varphi$ and any $\psi\in V_{\varphi}^{\bot}$ the action of $\mathbb{H}$ on $\varphi$ and $\psi$ spans all $\Delta$.

\section{Invariant description of a $SU(2)$--structure}

Let $(M,g)$ be a spin $5$--manifold with the corresponding Riemannian structure $g$. Denote by ${\rm Spin}(M)$ the spinor structure (with the structure group ${\rm Spin}(5)$) and let ${\bf S}$ be the associated spinor bundle, ${\bf S}={\rm Spin}(M)\times_{{\rm Spin}(5)}\Delta$, where $\Delta=\mathbb{C}^4$ is as in the first section. An $SU(2)$--structure on $M$ is a reduction $P$ of the frame bundle $SO(M)$ to the structure group $SU(2)\subset SO(5)$. We can extend $P$ to $P_{{\rm Spin}(5)}=P\times_{SU(2)}{\rm Spin}(5)$. Alternatively, as shown by Conti and Salamon \cite{CS}, an $SU(2)$--structure is given by the quadruplet $(\alpha,\omega_1,\omega_2,\omega_3)$ consisting of a $1$--form $\alpha$ and $2$--forms $\omega_1, \omega_2, \omega_3$ such that
\begin{align*}
\omega_k\wedge\omega_l=\delta_{kl}v\quad\textrm{and}\quad (X\lrcorner\omega_1=Y\lrcorner\omega_2\quad\Rightarrow\quad  \omega_3(X,Y)>0)
\end{align*} 
for some $4$--form $v$ satisfying $\alpha\wedge v \neq 0$. The third approach is the following \cite{CS}. Fix a unit spinor field $\varphi$. Then we define a subundle $P$ as a set of all frames $u$ such that $\varphi(x)=[u,\varphi_0]$, where $u$ is a frame over $x\in M$ and $\varphi_0\in \Delta$ is a fixed unit spinor. Then $P$ is an $SU(2)$ structure with $SU(2)={\rm Stab}(\varphi_0)$. 

Motivated by the final approach and the discussion in the first section we may consider the following approach to $SU(2)$--structures: Fix an admissible space $V\subset\Delta$, denote the corresponding Lie algebra by $\alg{su}(2)_V$ and its Lie group by $SU(2)_V$. In other words, $\alg{su}(2)_V=\alg{su}(2)_{\varphi_0}$ for any $\varphi_0\in V^{\bot}$ (see Theorem \ref{thm:su2equivalence}). If $P$ is an $SU(2)_V$--structure we may define a space (of real dimension $4$) of certain spinor fields
\begin{equation*}
{\bf S}_P=\{\varphi\in{\bf S}\mid \, \textrm{exists $\varphi_0\in V^{\bot}$ such that $\varphi=[u,\varphi_0]$ for any $u\in P$}\}.
\end{equation*}
The definition is correct since spinors from $V^{\bot}$ are fixed points of the action of $SU(2)_V$. It is clear that ${\bf S}_P$ is isomorphic to $V^{\bot}$ (and $V$). Any spinor field in ${\bf S}_P$ is said to {\it induce} given $SU(2)$--structure $P$.

By above considerations, there is a natural subbundle in the spinor bundle ${\bf S}$ over the $SU(2)$--structure. We call it adapted.

\begin{defin}
Assume $P$ is an $SU(2)$--structure on a spin manifold $M$. We say that a subbundle ${\bf V}$ in the spinor bundle ${\bf S}$ is {\it adapted} to $P$ if it is of the form ${\bf V}=P\times_{SU(2)}V$, where $V$ is admissible space in $\Delta$ such that $SU(2)=SU(2)_V$.
\end{defin}

Consider an almost complex structure on ${\bf S}$ induced by $j_1$ on $\Delta$ (see also discussion on the quaternionic structure below). From the definition it follows that adapted subbundle is complex $2$--dimensional.

Notice that two spinor fields $\varphi$ and $\psi$, which are sections of ${\bf V}^{\bot}$, i.e., an orthogonal complement of the adapted subbundle, do not in general induce the same $SU(2)$--structure. Indeed, assume $\varphi$ defines the underlying $SU(2)$--structure $P$, i.e., $\varphi=[u,\varphi_0]$, $u\in P$, where $SU(2)={\rm Stab}(\varphi_0)$. Then $\psi$ defines the same structure if and only if $\psi=[u,\psi_0]$, $u\in P$, for some $\psi_0$ in the admissible space $V$. In other words, $\varphi$ and $\psi$ must lie in the space ${\bf S}_P$ for some $P$ (which they induce).

We may consider a quaternionic structure $i_2$ on ${\bf S}$ induced from a quaternionic structure $i_2$ on $\Delta$. Therefore, $i_3=i_1\circ i_2$, where $i_1=j_1$, defines additional almost complex structure on ${\bf S}$ and $(i_1,i_2,i_3)$ forms a triple of almost complex structures \cite{BMM}. By Theorem \ref{thm:admissiblequaternionic} we have the following result.

\begin{cor}\label{cor:su2structurequaternionic}
An admissible subbundle ${\bf V}$ is quaternionic with respect to a quaternionic structure $(i_1,i_2,i_3)$.
\end{cor}

If an admissible subbundle ${\bf V}$ is fixed, we have, additionally, a quaternionic structure $(j_1,j_2,j_3)$ induced from $(i_1,i_2,i_3)$ on $\Delta$. These two structures differ only by a sign (for $i_2$ and $i_3$ on the orthogonal complement ${\bf V}^{\bot}$ of the admissible distribution).

\begin{cor}
An admissible distribution is quaternionic with respect to a quaternionic structure $(j_1,j_2,j_3)$.
\end{cor}

It is important to notice that all spinor fields $\varphi$, $j_1\varphi$, $j_2\varphi$ and $j_3\varphi$ (equivalently, $\varphi$, $i_1\varphi,i_2\varphi,i_3\varphi$) induce the same $SU(2)$--structure. In other words, if $\varphi\in {\bf S}_P$, then $j_k\varphi\in{\bf S}_P$ for any $k=1,2,3$. Indeed, each complex structure $j_k$ on $\Delta$ commutes with the action of ${\rm Spin}(5)$, i.e., $j_k(gs)=g(j_ks)$, $g\in{\rm Spin}(5)$, $s\in\Delta$. Therefore, if $\varphi$ is induced by the spinor $\varphi_0\in\Delta$, then $j_k\varphi$ is induced by $j_k\varphi_0$. Since the quaternionic structure $(j_1,j_2,j_3)$ leaves admissible space $V$ and its orthogonal complement $V^{\bot}$ invariant, it follows that $j_k\varphi$ induces the same $SU(2)$--structure as $\varphi$. Therefore if $\varphi$ is fixed, any spinor field $\psi$ defining the same $SU(2)$--structure is given by
\begin{equation}\label{eq:anyspinor}
\psi=a_0\varphi+a_1j_1\varphi+a_2j_2\varphi+a_3j_3\varphi=\sum_{k=0}^3 a_kj_k\varphi,\quad a_0,a_1,a_2,a_3\in\mathbb{R},
\end{equation}
where $j_0$ denotes the identity. Form this considerations, we also see that
\begin{equation*}
\nabla j_k=0,\quad k=1,2,3.
\end{equation*}

Let us relate how to derive a quadruplet $(\alpha,\omega_i)$ from the admissible distribution ${\bf V}$. A choice of an $SU(2)$--structure $P$ implies existence of a codimension one distribution ${\bf D}$, defined as ${\bf D}=P\times_{SU(2)}D$, where the existence of $D$ follows from the first section, and a unit orthogonal vector field $\zeta$ called {\it Reeb}. $\zeta$ is induced by a vector $y\in\mathbb{R}^5$ (see Lemma \ref{lem:repVphi}). Fix an admissible distribution ${\bf V}$ and consider the induced $SU(2)$--structure $P$. Denote by $\alg{su}(2)_+$ the Lie algebra dual to $\alg{su}(2)$ (compare the first section for details). Since the adjoint representation of ${\rm SU}(2)$ on $\alg{su}(2)_+$ is trivial, a bundle
\begin{equation*}
\alg{su}_+(M)=P\times_{{\rm SU}(2)}\alg{su}(2)_+
\end{equation*}
of certain $2$--forms is trivial. Hence, there are global linearly independent three $2$--forms $\omega_1,\omega_2,\omega_3$. Consider, moreover, a quaternionic structure $(J_1,J_2,J_3)$ on a distribution ${\bf D}$ described as follows (see subsection 1.7):
\begin{equation*}
J_k(X)\cdot\varphi_k=X\cdot j_1\varphi_k,
\end{equation*}
where $\varphi_k$ are three $\mathbb{R}$--linearly independent spinor fields in ${\bf V}^{\bot}$ (the representation of $SU(2)$ on ${\bf V}^{\bot}$ is trivial). 

\begin{prop}\label{prop:relationssu2definitions}
The forms $\alpha$ and $\omega_k$ defining $SU(2)$--structure in a sense of \cite{CS} may be given by the following relations
\begin{equation*}
\alpha=\zeta^{\flat},\quad \omega_k(X,Y)=g(J_k(X),Y).
\end{equation*}
\end{prop}
\begin{proof}
We may choose local section of orthonormal frame $(e_1,e_2,e_3,e_4,e_5)$ such that the quadruplet $(\alpha,\omega_k)$ defining $SU(2)$--structure in a sense of \cite{CS} is given by \cite{CS}
\begin{equation*}
\alpha=e_5^{\flat},\quad \omega_1=e_{12}+e_{34},\quad \omega_2=e_{13}-e_{24},\quad \omega_3=e_{14}+e_{23}.
\end{equation*}  
Then locally ${\bf V}^{\bot}=\cgtwo{s_1}{s_2}$ (compare the fundamental example in subsection 1.4). It suffices to choose $\varphi_1,\varphi_2,\varphi_3$ as in Remark \ref{rem:spinorsinVbot}.
\end{proof}

\begin{rem}\label{rem:bmmresult}
The relations contained in Proposition \ref{prop:relationssu2definitions}, adapted to the considered setting, have been already obtained in \cite{BMM}.
\end{rem}

\section{Characterization of intrinsic torsion modules}

In this section we want to derive the decomposition of the module of the space of all possible intrinsic torsions via spinorial approach. 

Let $M$ be a spin $5$--manifold with the corresponding Riemannian structure $g$. Let $SO(M)$ be a frame bundle of oriented orthonornal frames and ${\rm Spin}(M)\supset SO(M)$ the induced spin structure with the structure group ${\rm Spin}(5)$. The Levi--Civita connection $\nabla$ on $M$ induces a connection form $\omega$ on $SO(M)$ and $\tilde{\omega}$ on ${\rm Spin}(M)$. Let $P\subset SO(M)$ be an $SU(2)$--structure. Then an $\alg{su}(2)$--component of $\omega$ induces a Riemannian connection $\nabla^P$ on $M$. The {\it intrinsic torsion} of considered $SU(2)$--structure is a $(1,2)$--tensor field $\xi$ of the form
\begin{equation*}
\xi_XY=\nabla^P_XY-\nabla_XY.
\end{equation*}
From the definition of $\xi$ is follows that $\xi\in T^{\ast}M\otimes \alg{su}^{\bot}(M)$, where $\alg{su^{\bot}}(M)=P\times_{SU(2)}\alg{su}(2)^{\bot}$.

Moreover, $\tilde{\omega}$ and its $\alg{su}(2)$--component induce connections on the spinor bundle ${\bf S}={\rm Spin}(M)\times_{{\rm Spin}(5)}\Delta=P\times_{SU(2)}\Delta$. Denote them by $\nabla$ and $\nabla^P$ (as on $M$), respectively. If $\varphi\in {\bf S}$ is a spinor field defining $P$, it follows that $\nabla^P\varphi=0$.   

Let ${\bf V}$ be a subbundle adapted to $P$. Let $V$ be corresponding admissible space, $V^{\bot}$ is its orthogonal complement in $\Delta$. It is clear from the previous considerations, that with respect to the map $\omega\mapsto \omega\cdot\varphi_)$ for a fixed spinor $\varphi_0\in V^{\bot}$, we have an isomorphism of $\alg{su}(2)^{\bot}$ onto $\lin{\varphi_0}^{\bot}$. It can be shown \cite{CS} that with respect to this isomorphism
\begin{equation*}
\xi_X\cdot\varphi=-\nabla_X\varphi,
\end{equation*} 
where we consider $\xi_X$ as an element of $\alg{su}(2)^{\bot}$. More precisely, $\xi_X$ is treated as an invariant function from $P$ to $\alg{su}(2)^{\bot}$.

Denote by $\mathcal{T}$ the space $T^{\ast}(M)\otimes\alg{su}^{\bot}(M)$ of all possible intrinsic torsions. This space splits into irreducible modules under the action of the group $SU(2)$. In \cite{BMM}, the authors, applying the intuition developed in \cite{ACFH}, show how to rewrite $\nabla\varphi$ for a fixed unit spinor field $\varphi$ inducing considered $SU(2)$--structure into components lying in each irreducible component of $\mathcal{T}$. Let us recall this approach. Since for a unit spinor field $\varphi$, $\nabla_X\varphi$ is orthogonal to $\varphi$ is follows that there is a linear map $S^{\varphi}:TM\to {\bf D}$ and three one--forms $\beta^{\varphi}_k$ on $M$ such that
\begin{equation*}
\nabla_X\varphi=S^{\varphi}(X)\cdot\varphi+\sum_k \beta^{\varphi}_k(X)i_k\varphi.
\end{equation*} 
Here, $(i_1,i_2,i_3)$ is a quaternionic structure on ${\bf S}$, $i_2$ anticommutes with the multiplication by vectors, $i_1$ is induced by multiplication by the volume element \cite{BMM}. Let $S^{\varphi}(\zeta)=Z^{\varphi}$, where $\zeta$ is the Reeb field. Thus, we may write,
\begin{equation*}
S^{\varphi}=S^{\varphi}_{\bf D}+\alpha\otimes Z^{\varphi},
\end{equation*} 
where $\alpha$ is a one--form dual to $\zeta$ and $S_{\bf D}$ is an endomorphism of ${\bf D}$. Analogously, we may "decompose" each $\beta_k$ with respect to the splitting ${\bf D}\oplus\lin{\zeta}$ as
\begin{equation*}
\beta^{\varphi}_k=\beta^{\varphi,{\bf D}}_k+f^{\varphi}_k\alpha
\end{equation*}
for some function $f^{\varphi}_k$. Finally, $S^{\varphi}_{\bf D}$ splits as (we skip writing indices $\varphi$ and ${\bf D}$ to make the formula more readable) \cite{BMM}
\begin{equation*}
S^{\varphi}_{\bf D}=\lambda_0{\rm Id}_{\bf D}+S_0+\sum_k \lambda_k J_k+\sum_k \sigma_k,
\end{equation*}
where $S_0\in \alg{su}(M)=P\times_{SU(2)}\alg{su}(2)$ and $\sigma_k$ is such that $J_l\sigma_k=(-1)^{\delta_{kl}+1}\sigma_k J_l$, $l=1,2,3$. Thus elements $\lambda_0,\lambda_k,f_k,S_0,\sigma_k,\beta^{\varphi,{\bf D}}_k,Z^{\varphi}$ ($k=1,2,3$) are the components with respect to the splitting of $\mathcal{T}$ into irreducible modules \cite{BMM,CS}
\begin{equation*}
\mathcal{T}=7\mathbb{R}\oplus 4\alg{su}(2)_-\oplus 4(\mathbb{R}^4)^{\ast}\quad\textrm{pointwise}.
\end{equation*}
Here $\alg{su}(2)_-$ is a Lie algebra dual to $\alg{su}(2)$.

The aim is to make this construction independent of the choice of $\varphi$. Let ${\bf V}$ be a subbundle in ${\bf S}$ adapted to $P$ and let $V\subset\Delta$ be corresponding admissible subspace. We introduce a slight modification. Instead of considering a quaternionic structure $(i_1,i_2,i_3)$ and elements $\varphi,i_1\varphi,i_2\varphi,i_3\varphi$ spanning ${\bf V}$ we consider a quaternionic structure $(j_1,j_2,j_3)$. For a unit spinor field $\varphi$ we may write
\begin{equation}\label{eq:nablavarphiinvariant}
\nabla_X\varphi=S^{\varphi}(X)\cdot\varphi+\sum_k \beta^{\varphi}_k(X)j_k\varphi.
\end{equation}

The first main result of this section gives partial invariance.
\begin{prop}
The choice of $S^{\varphi}$ is independent on $\varphi$.
The one--forms $\beta^{\varphi}_k$ change with respect to the following formula: if $\psi$ is given by \eqref{eq:anyspinor}, then
\begin{align*}
\beta^{\psi}_1(X) &=(a_0^2+a_1^2-a_2^2-a_3^2)\beta^{\varphi}_1+2(a_1a_2-a_0a_3)\beta^{\varphi}_2+2(a_0a_2+a_1a_3)\beta^{\varphi}_3\\
\beta^{\psi}_2(X) &=2(a_1a_2+a_0a_3)\beta^{\varphi}_1+(a_0^2-a_1^2+a_2^2-a_3^2)\beta^{\varphi}_2+2(a_2a_3-a_0a_1)\beta^{\varphi}_3\\
\beta^{\psi}_3(X) &=2(a_1a_3-a_0a_2)\beta^{\varphi}_1+2(a_2a_3+a_0a_1)\beta^{\varphi}_2+(a_0^2-a_1^2-a_2^2+a_3^2)\beta^{\varphi}_3.
\end{align*}
\end{prop}
\begin{proof}
For $\varphi_0\in V$ denote by $(\varphi_0)_{\mathbb{H}}$ the following quadruplet of elements spanning $V$:
\begin{equation*}
(\varphi_0)_{\mathbb{H}}=(\varphi,j_1\varphi,j_2\varphi,j_3\varphi).
\end{equation*}
We have a natural action of $a=(a_0,a_1,a_2,a_3)\in\mathbb{H}$ on $V$ (compare subsection 1.8), namely
\begin{equation*}
a\cdot\varphi_0=\sum_k a_kj_k\varphi_0,
\end{equation*}
where $j_0$ is the identity. We have
\begin{equation*}
(a\cdot\varphi_0)_{\mathbb{H}}^{\top}=\rho(a)\cdot(\varphi_0)_{\mathbb{H}}^{\top},\quad
\rho(a)=\left(\begin{array}{cccc}
a_0 & a_1 & a_2 & a_3 \\
-a_1 & a_0 & -a_3 & a_2 \\
-a_2 & a_3 & a_0 & -a_1 \\
-a_3 & -a_2 & a_1 & a_0
\end{array}\right)
\end{equation*}
Notice that for $a\neq 0$, $\rho(a)^{-1}=\frac{1}{|a|^2}\rho(\bar{a})$ ($\rho$ is one of possible inclusions of $\mathbb{H}$ into $\alg{gl}(4,\mathbb{R})$).

Take two unit spinor fields $\varphi,\psi$ defining the same $SU(2)$--structure. Let $S^{\varphi},S^{\psi}$ and $\beta^{\varphi}_k,\beta^{\psi}_k$ be the corresponding elements with respect to the decomposition \eqref{eq:nablavarphiinvariant}. Firstly, we will show that $S^{\varphi}=S^{\psi}$. We have
\begin{align*}
\nabla_X\psi &=a_0\nabla_X\varphi+\sum_k a_k j_k(\nabla_X\varphi) \\
&=a_0S^{\varphi}(X)\cdot\varphi+a_0\sum_l \beta^{\varphi}_l(X)j_l\varphi\\
&+\sum_{k}a_k j_k(S^{\varphi}(X)\cdot\varphi)+\sum_{k,l} a_k\beta^{\varphi}_l(X)j_k(j_l\varphi).
\end{align*}
Since, by the definition of $j_k$, $j_k(S^{\varphi}(X)\cdot\varphi)=S^{\varphi}(X)\cdot j_k\varphi$, we get
\begin{align*}
\nabla_X\psi &=S^{\varphi}(X)\cdot\psi-\sum_k a_k\beta^{\varphi}_k(X)\varphi\\
&+(a_0\beta^{\varphi}_1(X)+a_2\beta^{\varphi}_3(X)-a_3\beta^{\varphi}_2(X))j_1\varphi\\
&+(a_0\beta^{\varphi}_2(X)+a_3\beta^{\varphi}_1(X)-a_1\beta^{\varphi}_3(X))j_2\varphi\\
&+(a_0\beta^{\varphi}_3(X)+a_1\beta^{\varphi}_2(X)-a_2\beta^{\varphi}_1(X))j_3\varphi.
\end{align*} 
Hence, $S^{\varphi}=S^{\psi}$. Moreover, by considerations at the beginning of the proof
\begin{equation*}
\varphi_{\mathbb{H}}=\rho(\bar{a})\psi_{\mathbb{H}}.
\end{equation*}
where $a=(a_0,a_1,a_2,a_3)$. Substituting this relation we get desired formula for the change of $\beta^{\varphi}_k$.
\end{proof}

Notice that components of each $\beta^{\psi}_k$ constitute the Hopf fibration (compare Proposition \ref{prop:quaternionic} and its proof).

Let us now show how to decompose $\nabla_X\varphi$ to obtain all "components" independent of $\varphi$. Recall that the multiplication of two--forms in $\alg{so}(5)$ by $\varphi_0\in\Delta$ is a surjective map onto $\lin{\varphi_0}^{\bot}$ with the kernel $\alg{su}(2)$. Moreover, restricted to $\alg{su}(2)_+$ (the dual to ${\alg su}(2)$) it is an isomorphism onto $V^{\bot}\cap\lin{\varphi_0}^{\bot}$, where $V$ is admissible space such that $\varphi_0\in V^{\bot}$. Since the action of $SU(2)$ on $\alg{su}(2)_+$ is trivial, for any spinor field $\varphi$ and a tangent vector $X$ there is a two--form $\omega^{\varphi}_X\in\alg{su}_+(M)$ such that
\begin{equation}\label{eq:nablavarphifullinvariance}
\nabla_X\varphi=S^{\varphi}(X)\cdot\varphi+\omega^{\varphi}_X\cdot\varphi,
\end{equation}
where $S^{\varphi}$ is as before and $\nabla\omega^{\varphi}_X=0$. Hence, $\omega\in T^{\ast}(M)\otimes\alg{su}_+(M)$. Comparing with \eqref{eq:nablavarphiinvariant} we have
\begin{equation*}
\omega^{\varphi}_X\cdot \varphi=\sum_k \beta^{\varphi}_k(X)j_k\varphi.
\end{equation*}
The advantage of the use of $\omega^{\varphi}$ is that we do not specify "coordinates" $(j_1\varphi,j_2\varphi,j_3\varphi)$ on the space ${\bf V}^{\bot}\cap\lin{\varphi}^{\bot}$.

\begin{prop}
Components $S^{\varphi}$ and $\omega^{\varphi}$ do not depend on $\varphi$.
\end{prop}
\begin{proof}
Fix a spinor field $\varphi$ and let spinor field $\psi$ be given by \eqref{eq:anyspinor}. Then
\begin{align*}
\nabla_X\psi &=a_0\nabla_X\varphi+\sum_k a_kj_k\nabla_X\varphi\\
&=a_0S^{\varphi}(X)\cdot\varphi+a_0\omega^{\varphi}_X\cdot\varphi+\sum_k a_kj_k(S^{\varphi}(X)\cdot\varphi)+\sum_ka_kj_k(\omega_X^{\varphi}\cdot\varphi)\\
&=a_0S^{\varphi}(X)\cdot\varphi+\sum_k a_kS^{\varphi}(X)\cdot j_k\varphi+a_0\omega^{\varphi}_X\cdot\varphi+\sum_k a_k\omega^{\varphi}_X\cdot j_k\varphi\\
&=S^{\varphi}(X)\cdot\psi+\omega^{\varphi}_X\cdot\psi.
\end{align*}
We used the fact that for any $k$, $j_k(Z\cdot\varphi)=Z\cdot j_k\varphi$, where $X\in TM$, and $j_k(\eta\cdot\varphi)=\eta\cdot j_k\varphi$, where $\eta$ is a two--form.
\end{proof}

By above Proposition we may skip writing index $\varphi$ in $S^{\varphi}$ and $\omega^{\varphi}$. Notice that $\omega$ from the formula \eqref{eq:nablavarphifullinvariance} splits as
\begin{equation*}
\omega_X=\omega_{X_{\bf D}}+\alpha(X)\omega_{\zeta}.
\end{equation*}
Then the first component belongs to the space ${\bf D}\otimes\alg{su}_+(M)$, whereas the second one to the space $\alg{su}_+(M)$.

\end{document}